\documentclass[12pt]{article}
\usepackage[a4paper,margin=0.82in]{geometry}
\usepackage[T1]{fontenc}
\usepackage{lmodern}
\usepackage{amsmath,amssymb,amsthm,mathtools}
\usepackage{microtype}
\usepackage[hidelinks]{hyperref}

\hypersetup{
  pdftitle={Convex-Hull Instability of the gamma2-Functional in Lp},
  pdfauthor={Helen W. J. Zhang and Chengdong Zhao},
  pdfsubject={Generic chaining, convexification, Lp spaces, and cotype},
  pdfkeywords={generic chaining, gamma2-functional, convex hulls, Lp spaces, cotype}
}

\newtheorem{theorem}{Theorem}[section]
\newtheorem{lemma}[theorem]{Lemma}
\newtheorem{proposition}[theorem]{Proposition}
\newtheorem{corollary}[theorem]{Corollary}

\newcommand{\conv}{\operatorname{conv}}
\newcommand{\diam}{\operatorname{diam}}
\newcommand{\Cprof}{\mathfrak C}

\title{Convex-Hull Instability of the $\gamma_2$-Functional in $L^p$}
\author{Helen W. J. Zhang\\
\small School of Mathematics, Hunan University\\
\small Changsha, Hunan 410082, China\\
\small \texttt{helenzhang@hnu.edu.cn}
\and
Chengdong Zhao\\
\small School of Mathematics and Statistics, Central South University\\
\small Changsha, Hunan 410083, China\\
\small \texttt{cdzhao@csu.edu.cn}}
\date{}

\begin{document}
\maketitle
\vspace{-1.35em}

\begin{abstract}
For every $2<p<\infty$ and every integer $r\ge2$, we construct a finite
set $T\subseteq S_{L^p[0,1]}$ such that $|T|\le2^{Cr^2}$,
$\gamma_2(T)\le Cr$, and
$\gamma_2(\conv T)\ge c r^{3/2-1/p}$.  Consequently, for every
fixed $p>2$, both estimates in Talagrand's Research Problem~2.11.3 fail
in each of the spaces $L^p[0,1]$ and $\ell_p$.  The lower bound follows
from a multilevel product principle applied to a scale-separated product
of Euclidean simplices.  The same construction gives the sharp
convexification profile of $\ell_p^r(\ell_2^{2^{128r}})$ for
$2<p\le\infty$ and, by finite representability, quantitative failure in
every infinite-dimensional Banach space with cotype index greater than
$2$.
\end{abstract}



\section{Introduction and main results}\label{sec:intro}

All Banach spaces are real.  For a positive integer $D$, write
$[D]=\{1,\ldots,D\}$.  We write $S_X$ and $B_X$ for the unit sphere and
unit ball of $X$, use the convention $1/\infty=0$, and let $\log$ and
$\log_2$ denote the natural and base-two logarithms.  The notation
$a\lesssim b$ means $a\le Cb$ for an absolute constant $C$; a subscript
indicates permitted dependence, and $a\asymp b$ means that both
inequalities hold.

We use Talagrand's partition normalization of $\gamma_2$.  An admissible
sequence on a metric space $(K,d)$ is a refining sequence of partitions
$(\mathcal A_n)_{n\ge0}$ such that $|\mathcal A_0|=1$ and
$|\mathcal A_n|\le2^{2^n}$ for $n\ge1$.  Thus
\[
 \gamma_2(K,d)=\inf_{(\mathcal A_n)}\sup_{x\in K}
 \sum_{n\ge0}2^{n/2}\diam\bigl(\mathcal A_n(x)\bigr).
\]
This is Talagrand's normalization
\cite[Definitions~2.7.1 and~2.7.3]{Talagrand2021}; see
\cite{Dudley1967,Fernique1975,Talagrand1996,Talagrand2021} for
background.

Talagrand's Theorem~2.11.1 gives
$\gamma_2(\conv T;H)\le L\gamma_2(T;H)$ for every finite subset $T$ of a
Hilbert space $H$, where $L$ is universal.  Research Problem~2.11.3 asks
whether every fixed $2$-smooth Banach space $X$ admits a constant $K_X$
such that $\gamma_2(\conv T;X)\le K_X\sqrt{\log|T|}$ for every finite
$T\subseteq B_X$, and, more generally, whether
$\gamma_2(\conv T;X)\le K_X\gamma_2(T;X)$ for every finite $T\subseteq X$
\cite[Theorem~2.11.1 and Research Problem~2.11.3]{Talagrand2021}.
For $N\ge2$, set
\[
 \Cprof_X(N)=
 \sup_{\substack{T\subset X\ \mathrm{finite}\\2\le |T|\le N}}
 \frac{\gamma_2(\conv T;X)}{\gamma_2(T;X)},
 \qquad
 \Cprof_X=\sup_{N\ge2}\Cprof_X(N).
\]
The denominator is positive because $\gamma_2(T)\ge\diam(T)$ whenever
$|T|\ge2$.

For $2<p\le\infty$ and an integer $r\ge2$, put $D=2^{128r}$ and
$Y_{p,r}=\ell_p^r(\ell_2^D)$.

\begin{theorem}\label{thm:main}
There are absolute constants $c,C>0$ such that, for every
$2<p\le\infty$ and every integer $r\ge2$, the unit sphere of $Y_{p,r}$
contains the explicit set
\[
 T_{p,r}=
 \left\{r^{-1/p}\sum_{i=1}^r e_{i,s_i}:1\le s_i\le D\right\},
\]
where $e_{i,s}$ is the $s$th standard unit vector in the $i$th copy of
$\ell_2^D$.  It satisfies
\[
 |T_{p,r}|=2^{128r^2},\qquad
 \gamma_2(T_{p,r};Y_{p,r})\le Cr,
\]
and
\[
 \gamma_2(\conv T_{p,r};Y_{p,r})\ge c\,r^{3/2-1/p}.
\]
\end{theorem}

Since $\sqrt{\log|T_{p,r}|}\asymp r$, both ratios in Talagrand's problem
are bounded below by a constant multiple of $r^{1/2-1/p}$.  For
$2<p<\infty$, each $Y_{p,r}$ embeds isometrically into $L^p[0,1]$ and is
finitely representable in $\ell_p$; Lemma~\ref{lem:transfer} normalizes
the latter embeddings without changing the order of either ratio.  As
$L^p[0,1]$ and $\ell_p$ are $2$-smooth, both estimates fail in each fixed
space.  Details are given in Section~\ref{sec:consequences}.

For the finite model, comparison with $H=\ell_2^r(\ell_2^D)$ gives
$\|x\|_{Y_{p,r}}\le\|x\|_H\le r^{1/2-1/p}\|x\|_{Y_{p,r}}$.
The Hilbert convex-hull inequality and Theorem~\ref{thm:main} therefore
yield $\Cprof_{Y_{p,r}}\asymp r^{1/2-1/p}$, including $p=\infty$.

Let $q_X=\inf\{q\ge2:X\text{ has Rademacher cotype }q\}$, with
$q_X=\infty$ when $X$ has no finite cotype.

\begin{corollary}\label{cor:cotype}
Let $X$ be infinite dimensional and suppose that $q_X>2$.  For every
sufficiently large $N$, there is a finite set $T\subseteq B_X$, with
$2\le|T|\le N$, such that
\[
 \min\left\{
 \frac{\gamma_2(\conv T;X)}{\sqrt{\log|T|}},
 \frac{\gamma_2(\conv T;X)}{\gamma_2(T;X)}
 \right\}
 \gtrsim (\log N)^{1/4-1/(2q_X)}.
\]
Consequently, neither of the two estimates stated above holds in $X$,
and $\Cprof_X=\infty$.
\end{corollary}

The proof combines a standard constant-weight packing estimate with a
multilevel product principle that forces one point to pay at several prescribed
admissible levels.  Related entropy and interpolation results for convex hulls
appear in \cite{Carl1997,CarlKyreziPajor1999,LiLinde2000,vanHandel2018}.

\section{A multilevel product principle}\label{sec:product}

\begin{proposition}\label{prop:product}
Let $(K,d)$ be a metric space.  For $1\le j\le m$, let $(S_j,d_j)$ be a
nonempty finite $\varepsilon_j$-separated metric space, where
$\varepsilon_j>0$.  Suppose that $K$ contains the full Cartesian product
$S=S_1\times\cdots\times S_m$ and that
$d_j(\pi_jx,\pi_jy)\le d(x,y)$ for all $x,y\in S$ and $1\le j\le m$.
If $n_1,\ldots,n_m$ are distinct nonnegative integers such that
\[
 \sum_{j=1}^m\frac{2^{2^{n_j}}}{|S_j|}<1,
\]
then
\[
 \gamma_2(K,d)\ge\sum_{j=1}^m2^{n_j/2}\varepsilon_j.
\]
\end{proposition}

\begin{proof}
Fix an admissible sequence $(\mathcal A_n)$ of partitions of $K$ and set
\[
 B_j=\{x\in S:\diam(\mathcal A_{n_j}(x))<\varepsilon_j\}.
\]
If a cell
$A\in\mathcal A_{n_j}$ has diameter less than $\varepsilon_j$, then all
points of $A\cap S$ have the same $j$th coordinate.  Hence
$|A\cap S|\le |S|/|S_j|$, and admissibility gives
$|B_j|\le2^{2^{n_j}}|S|/|S_j|$.  The assumed sum is less than one, so
some $x\in S\setminus\bigcup_jB_j$ exists; then
$\diam(\mathcal A_{n_j}(x))\ge\varepsilon_j$ for every $j$.  Since the
levels $n_j$ are distinct,
\[
 \sum_{n\ge0}2^{n/2}\diam\bigl(\mathcal A_n(x)\bigr)
 \ge\sum_{j=1}^m2^{n_j/2}\varepsilon_j.
\]
Taking the infimum over admissible sequences proves the result.
\end{proof}

\section{The mixed-norm construction}\label{sec:construction}

For every nonempty finite metric space $A$,
\[
 \gamma_2(A)\lesssim\diam(A)\sqrt{\log(1+|A|)}.
\]
Indeed, retain the one-cell partition until the first admissible level at
which the singleton partition is allowed, and use singletons thereafter.

\begin{proof}[Proof of Theorem~\ref{thm:main}]
The set $T_{p,r}$ lies in $S_{Y_{p,r}}$ and has cardinality
$D^r=2^{128r^2}$.  Its diameter is at most $2$, so the finite-set bound
gives $\gamma_2(T_{p,r})\lesssim r$.

Let $\Delta_D=\conv\{e_1,\ldots,e_D\}\subset\ell_2^D$.  Then
\begin{equation}\label{eq:conv-product}
 \conv T_{p,r}=r^{-1/p}(\Delta_D)^r.
\end{equation}
The forward inclusion is immediate.  For the reverse inclusion, if $x^{(i)}=\sum_s\lambda_{i,s}e_s\in\Delta_D$, then
\[
 r^{-1/p}(x^{(1)},\ldots,x^{(r)})
 =\sum_{s_1,\ldots,s_r}
 \left(\prod_{i=1}^r\lambda_{i,s_i}\right)
 r^{-1/p}\sum_{i=1}^re_{i,s_i},
\]
whose coefficients sum to one.

For $1\le i\le r$, set $k_i=2^{8i}$.  Since
$256\le k_i\le2^{8r}\le D/4$, choose a maximal family
\[
 \mathcal F_i\subseteq\{A\subseteq[D]:|A|=k_i\}
\]
whose distinct members satisfy $|A\triangle B|\ge k_i$.  The standard
packing--covering count in the constant-weight Hamming space gives
\[
 \binom D{k_i}
 \le |\mathcal F_i|
 \sum_{0\le s<k_i/2}\binom{k_i}{s}\binom{D-k_i}{s}.
\]
The sum on the right is at most
\[
 k_i2^{k_i}\left(\frac{2eD}{k_i}\right)^{k_i/2}
 \le\left(\frac{32eD}{k_i}\right)^{k_i/2},
\]
whereas $\binom D{k_i}\ge(D/k_i)^{k_i}$.  Hence
\[
 |\mathcal F_i|\ge\left(\frac{D}{32e\,k_i}\right)^{k_i/2}.
\]
Inside the $i$th scaled simplex $r^{-1/p}\Delta_D$, define
\[
 S_i=\left\{\frac{r^{-1/p}}{k_i}\sum_{s\in A}e_s:
 A\in\mathcal F_i\right\}.
\]
For distinct $A,B\in\mathcal F_i$,
\[
 \left\|\frac{r^{-1/p}}{k_i}\sum_{s\in A}e_s
 -\frac{r^{-1/p}}{k_i}\sum_{s\in B}e_s\right\|_2
 =\frac{r^{-1/p}}{k_i}|A\triangle B|^{1/2}
 \ge r^{-1/p}k_i^{-1/2}=: \varepsilon_i.
\]
Moreover, since $\log_2(32e)<7$ and $i\le r$,
\[
 \log_2|S_i|
 \ge\frac{k_i}{2}\bigl(128r-8i-\log_2(32e)\bigr)
 \ge50rk_i.
\]

Put $n_i=\lfloor\log_2(4rk_i)\rfloor$.  Then
$n_{i+1}=n_i+8$ and $2rk_i<2^{n_i}\le4rk_i$, so
\[
 \frac{2^{2^{n_i}}}{|S_i|}\le2^{-46rk_i},\qquad
 \sum_{i=1}^r\frac{2^{2^{n_i}}}{|S_i|}
 \le r\,2^{-46rk_1}<1.
\]
By \eqref{eq:conv-product}, the full product
$S_1\times\cdots\times S_r$ lies in $\conv T_{p,r}$.  With the
Euclidean metric on each $S_i$, the coordinate projections are
nonexpansive for the $\ell_p$-sum norm.  Proposition~\ref{prop:product}
therefore gives
\[
 \gamma_2(\conv T_{p,r})
 \ge\sum_{i=1}^r2^{n_i/2}\varepsilon_i
 >\sum_{i=1}^r(2rk_i)^{1/2}r^{-1/p}k_i^{-1/2}
 =\sqrt2\,r^{3/2-1/p}.
\]
\end{proof}

Finally, if $H=\ell_2^r(\ell_2^D)$, then
$\|x\|_{Y_{p,r}}\le\|x\|_H\le
r^{1/2-1/p}\|x\|_{Y_{p,r}}$.  The Hilbert convex-hull inequality gives
$\Cprof_{Y_{p,r}}\lesssim r^{1/2-1/p}$, while
Theorem~\ref{thm:main} gives the reverse bound.  Hence
$\Cprof_{Y_{p,r}}\asymp r^{1/2-1/p}$.

\section{Transfer and consequences}\label{sec:consequences}

A Banach space $Z$ is finitely representable in $X$ if every
finite-dimensional subspace of $Z$ embeds into $X$ with distortion
arbitrarily close to $1$.  Finite representability is transitive.

\begin{lemma}\label{lem:transfer}
Let $T\subset Z$ be finite, and suppose that a linear map
$U:\operatorname{span}T\to X$ satisfies
$m\|z\|\le\|Uz\|\le M\|z\|$ for some $m,M>0$.  Put
$S=M^{-1}UT$.  Then $|S|=|T|$ and
\[
 \gamma_2(S;X)\le\gamma_2(T;Z),\qquad
 \gamma_2(\conv S;X)\ge\frac mM\gamma_2(\conv T;Z).
\]
If $T\subseteq B_Z$, then $S\subseteq B_X$.
\end{lemma}

\begin{proof}
The assertions follow from metric comparison and
$U(\conv T)=\conv(UT)$.
\end{proof}

Consequently, if $Z$ is finitely representable in $X$, then
$\Cprof_X(N)\ge\Cprof_Z(N)$ for every $N\ge2$; the same argument
transfers unit-ball witnesses for the entropy-normalized estimate.

Fix $2<p<\infty$.  Let $G_1,\ldots,G_D$ be independent standard real
Gaussians and put $c_p=(\mathbb E|G_1|^p)^{1/p}$.  Partition $[0,1]$
into intervals $I_1,\ldots,I_r$ of positive length and, on each $I_i$
with normalized Lebesgue measure, realize independent standard Gaussians
$g_{i,1},\ldots,g_{i,D}$.  Define
\[
 (\Phi x)(t)=\frac{|I_i|^{-1/p}}{c_p}
 \sum_{s=1}^D x_s^{(i)}g_{i,s}(t),\qquad t\in I_i.
\]
Rotational invariance gives
\[
 \|\Phi x\|_p^p
 =\sum_{i=1}^r c_p^{-p}\mathbb E
   \left|\sum_{s=1}^D x_s^{(i)}G_s\right|^p
 =\sum_{i=1}^r\|x^{(i)}\|_2^p.
\]
Thus $\Phi:Y_{p,r}\to L^p[0,1]$ is a linear isometry.

Given a finite-dimensional $E\subset L^p[0,1]$ and $\eta>0$, approximate
a basis of $E$ sufficiently closely by simple functions on a common
finite partition.  The resulting map has distortion at most $1+\eta$,
and its range is isometric to a subspace of a weighted $\ell_p^M$,
which is isometric to $\ell_p^M$.  Thus $L^p[0,1]$ is finitely representable in $\ell_p$.

Let $S_r=\Phi(T_{p,r})$.  Then $S_r\subseteq S_{L^p}$ and
\[
 \min\left\{
 \frac{\gamma_2(\conv S_r)}{\sqrt{\log|S_r|}},
 \frac{\gamma_2(\conv S_r)}{\gamma_2(S_r)}
 \right\}
 \gtrsim r^{1/2-1/p}\longrightarrow\infty.
\]
Applying Lemma~\ref{lem:transfer} to embeddings into $\ell_p$ gives
unit-ball sets with the same divergent lower bound.  Pinelis' inequality
$\|x+y\|^2+\|x-y\|^2\le2\|x\|^2+2(p-1)\|y\|^2$ shows that both
$L^p[0,1]$ and $\ell_p$ are $2$-smooth
\cite[Proposition~2.1]{Pinelis1994}.  Hence both estimates in Research
Problem~2.11.3 fail in each fixed space.

\begin{proof}[Proof of Corollary~\ref{cor:cotype}]
If $2<q_X<\infty$, the Maurey--Pisier endpoint theorem gives finite
representability of $\ell_{q_X}$ in $X$
\cite[Remark~0.2 and Theorem~1.1, pp.~50, 54--55]{MaureyPisier1976}.
The preceding Gaussian and discretization arguments show that each
$Y_{q_X,r}$ is finitely representable in $\ell_{q_X}$, hence in $X$.

If $q_X=\infty$, the theorem of Bastero--Uriz gives finite
representability of $c_0$ in $X$ \cite[Theorem~1]{BasteroUriz1986}.
A finite norming set embeds every finite-dimensional Banach space with
distortion arbitrarily close to $1$ into some $\ell_\infty^M$, and
$\ell_\infty^M$ embeds isometrically into $c_0$.  Hence each $Y_{\infty,r}$ is finitely
representable in $X$.

Apply Lemma~\ref{lem:transfer} to the corresponding model set and take
$r=\lfloor(\log_2N/128)^{1/2}\rfloor$.  For sufficiently large $N$,
$|T_{p,r}|\le N$ and $r\asymp\sqrt{\log N}$, while both model ratios are
$\gtrsim r^{1/2-1/p}$, with $p=q_X$ in the first case and $p=\infty$ in
the second.  The stated estimate follows from $1/\infty=0$.
\end{proof}

\paragraph{The cotype-$2$ endpoint.}
Talagrand also asked in private correspondence whether bounded
convexification profile characterizes Hilbert space.  In the $2$-smooth
class, Corollary~\ref{cor:cotype} shows that $\Cprof_X<\infty$ forces
$q_X=2$.  Since $2$-smoothness implies Rademacher type $2$, Kwapie\'n's
theorem yields Hilbertian structure if $X$ also has cotype $2$
\cite[Corollary~3.2]{Kwapien1972}.  The remaining issue is that equality
of the cotype index need not assert that cotype $2$ is attained.

\section*{Acknowledgments}
The authors thank Michel Talagrand for his encouragement, for pointing out
the extension of the initial example to $L^p$, for suggesting the
Hilbert-rigidity question mentioned above, and for comments that led to a
substantial simplification of the manuscript.

\paragraph{Use of generative AI.}
The initial counterexample and its core product-of-simplices construction
were found by the authors.  Generative AI tools were used in the subsequent
development of the manuscript, including exploratory discussion of extensions
and consequences, checking intermediate arguments, and language and \LaTeX{}
editing.  All outputs were independently assessed and verified; the authors
take full responsibility for the paper.

\end{document}